\documentclass{article}

\usepackage{arxiv}
\usepackage{amsmath,amsopn}
\usepackage{float}
\usepackage{enumerate}
\usepackage[utf8]{inputenc} 
\usepackage[T1]{fontenc}    
\usepackage[hidelinks]{hyperref}      
\usepackage{url}            
\usepackage{booktabs}       
\usepackage{amsfonts}       
\usepackage{ stmaryrd }

\usepackage{soul}

\usepackage{graphicx}
\usepackage{dcolumn}
\usepackage{bm}
\usepackage{xspace}
\usepackage{mathtools}
\usepackage{amssymb}
\usepackage{amsthm}

\usepackage{xcolor}
\usepackage[mathscr]{euscript}

\theoremstyle{plain}
\newtheorem{theorem}{Theorem}[section]

\newtheorem{lemma}[theorem]{Lemma}

\newtheorem*{theorem*}{Theorem}

\theoremstyle{definition}

\theoremstyle{remark}

\newcommand{\R}{\mathbb{R}}
\newcommand{\Z}{\mathbb{Z}}
\newcommand{\C}{\mathbb{C}}
\newcommand{\Q}{\mathbb{Q}}
\newcommand{\abs}[1]{\left|#1\right|}

\renewcommand{\pmod}[1]{\hspace{0.25em}\left(\operatorname{mod}\text{ }#1\right)}

\newcommand{\slz}{\ensuremath{\operatorname{SL}_2(\mathbb{Z})}\xspace}
\newcommand{\abcd}{\ensuremath{\left(\begin{smallmatrix}a&b\\c&d\end{smallmatrix}\right)}\xspace}

\newcommand{\mpz}{\ensuremath{\operatorname{Mp}_2(\mathbb{Z})}\xspace}
\DeclareMathOperator{\Mp}{Mp}
\newcommand{\mpr}{\ensuremath{\operatorname{Mp}_2(\mathbb{R})}\xspace}

\DeclareMathOperator{\spn}{span}
\DeclareMathOperator{\SL}{SL}
\DeclareMathOperator{\GL}{GL}
\newcommand{\pfrac}[2]{\left(\frac{#1}{#2}\right)}

\newcommand{\ep}{\varepsilon}

\mathtoolsset{showonlyrefs}

\renewcommand{\pmatrix}[4]{\left(\begin{smallmatrix}#1 & #2 \\ #3 & #4\end{smallmatrix}\right)}

\title{An infinite family of vector-valued mock theta functions}

\author{ 
	{Nickolas Andersen}\\
	Department of Mathematics\\
	Brigham Young University\\
	Provo, UT \\
	\texttt{nick@math.byu.edu} \\
	\And
	{Clayton Williams}\\
	Department of Mathematics\\
	University of Illinois, Urbana-Champaign\\
	Urbana, IL \\
	\texttt{cw78@illinois.edu} \\
}

\renewcommand{\headeright}{}
\renewcommand{\undertitle}{}
\renewcommand{\shorttitle}{Vector-valued mock theta functions}

\hypersetup{
pdftitle={Vector-Valued Mock Theta Functions},
pdfsubject={math.NT},
pdfauthor={Nick Andersen, Clayton Williams},
pdfkeywords={mock theta functions, weil representation, vector-valued forms, modular forms, maass forms, mock modular forms},
}

\begin{document}
\maketitle

\begin{abstract}
We exhibit an infinite family of vector-valued mock theta functions indexed by positive integers coprime to $6$.
These are built from specializations of Dyson's rank generating function and related functions studied by Watson, Gordon, and McIntosh.
The associated completed harmonic Maass forms transform according to the Weil representation attached to a rank one lattice.
This strengthens a 2010 result of Bringmann and Ono and a 2019 result of Garvan.
\end{abstract}

\section{Introduction}

Ramanujan introduced the mock theta functions in his famous last letter to G.H. Hardy in January 1920, giving 17 examples together with several conjectural relations involving them (see Section~14 of \cite{andrews-berndt-lost-notebook-5}).
He grouped these functions into families which he named third order, fifth order, and seventh order, where the conjectured identities relate functions of the same order.
For example, three of the fifth order mock theta functions are
\begin{align*}
    \phi_0(q) &= 1 + \sum_{n=1}^\infty q^{n^2}(1+q)(1+q^3)\cdots(1+q^{2n-1}), \\
    \chi_0(q) &= 1 + \sum_{n=1}^\infty \frac{q^n}{(1-q^{n+1})\cdots(1-q^{2n})}, \\
    F_0(q) &= 1 + \sum_{n=1}^\infty \frac{q^{2n^2}}{(1-q)\cdots(1-q^{2n-1})},
\end{align*}
and they are related by
\begin{equation} \label{eq:phi0-chi0-F0}
    \phi_0(-q) + \chi_0(q) = 2F_0(q).
\end{equation}
Watson proved \eqref{eq:phi0-chi0-F0}, along with all of the other fifth order relations in Ramanujan's last letter, in \cite{watson}.

Ramanujan also listed ten identities (called the ``mock theta conjectures'') in the Lost Notebook (see Sections~5--6 of \cite{andrews-berndt-lost-notebook-5}), each involving one of the fifth order mock theta functions and a specialization of the function
\begin{equation} \label{eq:mrq-def}
    M(r,q) = \sum_{n=1}^\infty \frac{q^{n(n-1)}}{(q^r;q)_n(q^{1-r};q)_n}, \qquad (a;q)_n = \prod_{m=0}^{n-1}(1-aq^m).
\end{equation}
See \cite{andrews-garvan} and \cite{gordon-mcintosh} for more  on the mock theta conjectures.
One of the identities is
\begin{equation}
    F_0(q) - 1 = qM(\tfrac 15,q^5) - q\psi(q^5) H(q^2),
\end{equation}
where $\psi(q)$ is a theta function and $H(q)$ is one of the Rogers-Ramanujan functions.
The mock theta conjectures were proved by Hickerson \cite{hickerson1988proof}, who used Andrews and Garvan's result \cite{andrews-garvan} that the identities naturally group into two families of five and that in each family, all five identities are equivalent.
Zwegers showed in his PhD thesis \cite{zwegers2002mock} that Ramanujan's 17 mock theta functions can be completed to real analytic vector-valued modular forms of weight $1/2$, thus the mock theta conjectures can be interpreted as identities among harmonic Maass forms.
Folsom \cite{folsom} used this framework to reduce the mock theta conjectures to a large finite computation. 
More recently, the first author \cite{andersen2016vector} gave a simultaneous proof of eight of the ten identities (the remaining two following easily from the rest) by proving an equality between real analytic vector-valued modular forms that transform according to the Weil representation. 
This approach does not require a large computational verification.
It also provides a conceptual reason for the identities to be true and for the identities within the Andrews--Garvan families to be related.

The main objects of study in this paper are the mock theta functions $M(r,q)$ and their interactions with the rank generating function
\begin{equation} \label{eq:rank-gen-def}
    R(\zeta,q) = 1 + \sum_{n=1}^\infty \frac{q^{n^2}}{(\zeta q;q)_n(\zeta^{-1}q;q)_n} = 1 + \sum_{n=1}^\infty \sum_{m=-\infty}^\infty N(m,n)\zeta^mq^n.
\end{equation}
Here $N(m,n)$ is the number of partitions of $n$ with rank $m$ (see \cite{bringono2010dyson}, for example).
The function $R(-1,q)$ equals $f(q)$, one of Ramanujan's third order mock theta functions; here we will consider the cases when $\zeta=\zeta_c^a=e^{2\pi i a/c}$.
Bringmann and Ono \cite{bringono2010dyson} gave a recipe for completing the functions $M(\frac{a}{c},q)$ and $R(\zeta_c^a,q)$, where $(c,6)=1$, to harmonic Maass forms of weight $1/2$, and described their transformation properties under the generators $z\mapsto z+1$ and $z\mapsto -1/z$ of the modular group.
Here $q$ and $z$ are related in the usual way: $q=e(z)=e^{2\pi iz}$.
Garvan \cite{garvan2019transformation} also studied these functions and made the formulas of Bringmann and Ono more explicit.
For example, he showed that the functions
\begin{align}
    \mathcal M(a,0,c;z) &= 2q^{\frac{3a}{2c}(1-\frac ac)-\frac 1{24}} M(\tfrac ac,q) + \varepsilon(\tfrac ac,z) + \mathrm{NH}_1(z) \\
    \text{and }\quad \mathcal N(a,0,c;z) &= \csc (\tfrac{\pi a}{c}) q^{-\frac 1{24}} R(\zeta_c^a,q) + \mathrm{NH}_2(z)
\end{align}
are harmonic Maass forms of weight $1/2$, where, for $j\in \{1,2\}$, $\mathrm{NH}_j$ is a nonholomorphic period integral of a unary theta function, and $\varepsilon(\frac ac,z)$ is only nonzero if $0\leq \frac{a}{c}<\frac 16$ or $\frac 56<\frac ac<1$, in which case it is a multiple of a power of $q$.
The transformation laws for $\mathcal M(a,0,c;z)$ and $\mathcal N(a,0,c;z)$ under $z\mapsto z+1$ and $z\mapsto -1/z$ introduce auxiliary functions $\mathcal M(a,b,c;z)$ and $\mathcal N(a,b,c;z)$, where $0\leq a,b<c$.
See Section~2 for precise definitions of these functions.
By Theorem~3.1 of \cite{garvan2019transformation} we have
\begin{align}\label{eq:M-T-transform}
    \mathcal M(a,b,c;z+1) &= \zeta_{2c}^{5a}\zeta_{2c^2}^{-3a^2}\zeta_{24}^{-1}
    \mathcal M(a,[a+b]_c,c;z), \\ \label{eq:N-T-transform}
    \mathcal N (a,b,c;z+1)&=
    \zeta_{2c^2}^{3b^2}\zeta_{24}^{-1}\mathcal N ([a-b]_c,b,c;z)\times \begin{dcases}
    1 &\text{if }a\geq b,\\
    -\zeta_c^{-3b}&\text{otherwise,}
    \end{dcases}
\end{align}
where $[n]_c$ denotes the least nonnegative residue of $n$ modulo $c$,
and
\begin{equation} \label{eq:M-N-S-transform}
  \mathcal M\left(a,b,c;-1/z\right)=
  \sqrt{-iz} \, \mathcal N (a,b,c;z).
\end{equation}
From these transformation laws it follows that
for each positive integer $c$ coprime to $6$,
\begin{equation}
    \spn\left(\{\mathcal M(a,b,c;z):0\leq a,b<c\}\cup \{\mathcal N(a,b,c;z):0\leq a,b<c\}\right)
\end{equation}
is closed under the action of $\SL_2(\Z)$; this is essentially the statement of Theorem~3.4 of \cite{bringono2010dyson} and Corollary~3.2 of \cite{garvan2019transformation}.

Our theorem makes this observation more explicit by describing a vector-valued harmonic Maass form of weight $1/2$ that transforms according to the Weil representation attached to a finite quadratic module of order $12c^2$.
See Section~\ref{sec:weil-rep} for precise definitions and notation. 
Briefly, let $L(c)$ denote the lattice $\Z$ together with the bilinear form $(x,y)=-12c^2xy$.
If $L'(c)$ denotes the dual lattice, then we have $L'(c)/L(c)\cong \Z/12c^2\Z$.
Denote the standard basis of $\C[L'(c)/L(c)]$ by $\{\mathfrak e_h: h\in \Z/12c^2\Z\}$  and let $\rho_{L(c)}:\Mp_2(\Z)\to \GL(\C[L'/L])$ be the associated Weil representation.
Finally, let $\mathcal H_k(\rho_{L(c)})$ denote the space of harmonic Maass forms of weight $k$ and type $\rho_{L(c)}$.

\begin{theorem}\label{thm:main}
For each positive integer $c$ coprime to $6$, define
\begin{equation*}
    H_c(z) = \sum_{h(12c^2)}\sum_{a=0}^{c-1}\sum_{b=0}^{c-1} \big(\alpha_h(a,b) \mathcal M(a,b,c;z) + \beta_h(a,b)\mathcal N(a,b,c;z)\big)\mathfrak e_h,
\end{equation*}
where
\begin{align*}
    \alpha_h(a,b) &=
    \begin{dcases}
    \pm \frac{1}{2\sin(\pi/c)} \pfrac{12}{k} e\left(\frac{b(5\pm k)}{2c}\right) & \text{ if }a\neq 0 \text{ and } h = \pm 6a + ck, \\
    \frac{i\sin(b k\pi /c)}{\sin(\pi/c)}\pfrac{12}{k}e\left(\frac{5b}{2c}\right) & \text{ if }a=0 \text{ and }h=ck, \\
    0 & \text{ otherwise,}
    \end{dcases}
    \\
    \beta_h(a,b) &= 
    \begin{dcases}
    \pm \frac{i}{2\sin(\pi/c)} \pfrac{12}{k} e\left(\frac{5b\pm a k}{2c}-\frac{3ab}{c^2}\right) & \text{ if }b\neq 0 \text{ and } h = \mp 6b + ck, \\
    \frac{-\sin(a k\pi /c)}{\sin(\pi/c)}\pfrac{12}{k} & \text{ if }b=0\text{ and }h=ck, \\
    0 & \text{ otherwise.}
    \end{dcases}
\end{align*}
Then ${H}_c (z) \in \mathcal H_{\frac 12}(\rho_{L(c)})$. 
\end{theorem}

Here, and throughout, $\pfrac{12}{\cdot}$ denotes the Kronecker symbol. 
We have chosen to normalize $H_c$ as above because the quantity
$\frac{\sin(bk\pi/c)}{\sin(\pi/c)}$
is a cyclotomic unit when $c$ is a prime power and $(bk,c)=1$ (see \cite[Chapter~8]{washington}). 

In Section~\ref{sec:weil-rep} we give some brief background on the Weil representation and harmonic Maass forms, and we summarize the relevant results of \cite{bringono2010dyson} and \cite{garvan2019transformation}.
We prove Theorem~\ref{thm:main} in Section~\ref{sec:proof}.
While the proof is a straightforward verification of a few exponential sum identities, the discovery of the formula in Theorem~\ref{thm:main} was less straightforward.
In Section~\ref{sec:motivation} we briefly discuss how we arrived at Theorem~\ref{thm:main} in the hope that such techniques might be useful in future work.

\subsubsection*{Acknowledgements}
The first author is supported by the Simons Foundation, award number 854098. A result similar to Theorem~\ref{thm:main} for $c=7$ first appeared in the second author's master's thesis \cite{williams_vvmocktheta_2022}.

\section{The Weil representation and harmonic Maass forms} \label{sec:weil-rep}

Let $\mathbb H$ denote the complex upper half-plane and let $\operatorname{SL}_2(\mathbb{R})$ act on $\mathbb{H}$ in the usual way via
$\abcd z = \frac{az+b}{cz+d}$.
The metaplectic group $ \Mp_2(\R) $ is the set
\begin{align}
\mpr =\left\{(M,\phi(z)):M=\abcd\in \operatorname{SL}_2(\R), \textrm{ }\phi(z)^2 = cz+d \textrm{ and }\phi\textrm{ is a holomorphic function}\right\}
\end{align}
together with the group operation
\begin{align}
   (M_1,\phi_1(z))(M_2,\phi_2(z)) = (M_1M_2,\phi_1(M_2z)\phi_2(z)).
\end{align}
Note $\mpr$ is a double-sheeted cover of $\operatorname{SL}_2(\mathbb{R})$. Let $\mpz$ be the inverse image of $\slz$ under the covering map. 
It is generated by $T=(\pmatrix 1101,1)$ and $S=(\pmatrix 0{-1}10,\sqrt{z})$.

Let $ L $ be a lattice with a nondegenerate symmetric bilinear form $ (\cdot,\cdot):L\times L\to \Z$ whose associated quadratic form $ q(x) = \frac{1}{2}(x,x) $ takes its values in $\Z$.
The dual lattice is
\[
L' = \{x\in L\otimes_\Z\Q:(x,y)\in\Z \textrm{ for all }y\in L\} .
\]
The quotient group $ L'/L $ is a finite abelian group called the discriminant group.
The vector space $ \C[L'/L] $ has a standard basis conventionally denoted by $ \{\mathfrak{e}_\gamma:\gamma\in L'/L\}. $ 
Let $(n_-,n_+)$ be the signature of $L$.

There is a representation $\rho_L:\Mp_2(\Z)\to\GL(\C[L'/L])$ called the Weil Representation; it is defined on the generators $T$ and $S$ by 
   \begin{align}
     \rho_L(T)\mathfrak{e}_\gamma &= e(q(\gamma))\mathfrak{e}_\gamma, \label{eq:weil-T}
     \\
     \rho_L(S)\mathfrak{e}_\gamma &= \frac{1}{\sqrt{\abs{L'/L}}}e\left(\frac{n_--n_+}{8}\right)\sum_{\gamma\in L'/L}e(-(\gamma,\gamma'))\mathfrak{e}_{\gamma'}. \label{eq:weil-S}
   \end{align}
For more on the Weil representation, see \cite[pages 15-16]{bruinier2002borprod} and the references therein.

For $k\in \R$, the weight $ k $ hyperbolic Laplacian is defined by
  \begin{align*}
     \Delta_k = -y^2\left(\frac{\partial^2}{\partial x^2}+\frac{\partial^2}{\partial y^2}\right)+iky\left(\frac{\partial}{\partial x}+i\frac{\partial}{\partial y}\right), \qquad z = x+ iy.
  \end{align*}
A function $ f:\mathbb{H}\to \mathbb{C}[L'/L] $ is a vector-valued weak Maass form of weight $ k \in \frac 12\Z $ and type $\rho_L$ if
   \begin{enumerate}
\item $ f(Mz) = \phi(z)^{2k}\rho_L((M,\phi))f(z)$ for all $ (M,\phi)\in \mpz $,
\item there exists a constant $ \lambda $ such that, for all $ z $, $ \Delta_k f(z)=\lambda f(z) $ ($\Delta_k$ is applied componentwise), and

      \item $ f $ has at most linear exponential growth at the cusp at $ \infty. $
\end{enumerate}
A harmonic Maass form is a weak Maass form with $ \lambda=0 $, and the holomorphic part of a harmonic Maass form is called a mock modular form.
For more on (scalar-valued) harmonic Maass forms, see \cite{bfor}.

\subsection{The harmonic Maass forms \texorpdfstring{$\mathcal M(a,b,c;z)$}{M(a,b,c;z)} and \texorpdfstring{$\mathcal N(a,b,c;z)$}{N*a,b,c;z)}}

Here we record the definitions and transformation formulas of the families of harmonic Maass forms used to build $H_c(z)$ in Theorem~\ref{thm:main}. 
Most of this material can be found in Sections~2 and 3 of \cite{garvan2019transformation} (see also \cite{bringono2010dyson}).
Our notation is slightly different than that of \cite{garvan2019transformation}; notably we use $ \mathcal{N},\mathcal{M} $ instead of $ \mathcal{G}_1,\mathcal{G}_2 $ 
and we write $\mathcal N(a,0,c;z)$ and $\mathcal M(a,0,c;z)$ for $\mathcal G_1(\frac ac;z)$ and $\mathcal G_2(\frac ac;z)$ respectively.

Let $ c $ be a positive integer coprime to $6$. 
Generalizing \eqref{eq:mrq-def}, let
\begin{align}
  M(a,b,c;z)&=\frac{1}{(q;q)_\infty}\sum_{n=-\infty}^\infty \frac{(-1)^nq^{n+a/c}}{1-\zeta_c^b q^{n+a/c}}q^{\frac{3}{2}n(n+1)},
\end{align}
where $0\leq a,b<c$ and $a,b$ are not both zero.
To generalize \eqref{eq:rank-gen-def}, we first define
\begin{align}
    k(b,c)&=\begin{dcases}
        0&\textrm{ if }0\leq\tfrac{b}{c}<\tfrac{1}{6},\\
        1&\textrm{ if }\tfrac{1}{6}<\tfrac{b}{c}<\tfrac{1}{2},\\
        2&\textrm{ if }\tfrac{1}{2}<\tfrac{b}{c}<\tfrac{5}{6},\\
        3&\textrm{ if }\tfrac{5}{6}<\tfrac{b}{c}<1,\\\end{dcases}
\end{align}
and
\begin{align*}
  K(a,b,c,n;z)&=(-1)^n\frac{ \sin \left(\frac{\pi  a}{c}-\pi  z \left(2 n
  k(b,c)+\frac{b}{c}\right)\right)+q^n \sin \left(\frac{\pi  a}{c}-\pi  z \left(\frac{b}{c}-2 n k(b,c)\right)\right)}{1-2 q^n \cos \left(\frac{2 \pi  a}{c}-\frac{2 \pi  b z}{c}\right)+q^{2 n}}.
\end{align*}
Note that
\begin{equation}
    K(a,0,c, n;z) = \frac{(-1)^n \left(1+q^n \right)\sin \left(\frac{\pi  a}{c} \right)}{1-2 q^n \cos \left(\frac{2 \pi  a}{c}\right)+q^{2 n}}. 
\end{equation}
For $0\leq a,b<c$ with $a,b$ not both zero, we define
\begin{align}
  N(a,b,c;z)=\frac{1}{(q;q)_\infty}\left(\frac{i\zeta_{2c}^{-a}q^{b/2c}}{2(1-\zeta_c^{-a}q^{b/c})}+\sum_{n=1}^\infty K(a,b,c,n;z)q^{\frac{n(3n+1)}{2}}\right).
\end{align}
By equation (1.6) of \cite{gordon-mcintosh} we have
\begin{equation}
    N(a,0,c;z) = \frac 14\csc\left(\frac{\pi a}{c}\right) R(\zeta_c^a;q).
\end{equation}
Further, define
\begin{align}
  \varepsilon(a,b,c;z)=\begin{dcases}
    2\zeta_c^{-2b} q^{-\frac 32(\frac ac-\frac 16)^2} & \text{if  } 0\leq\tfrac{a}{c}<\tfrac{1}{6},\\
    0 & \text{if }\tfrac{1}{6}<\tfrac{a}{c}<\tfrac{5}{6},\\
    2 q^{-\frac 32(\frac ac - \frac 56)^2} & \text{if  } \tfrac{5}{6}<\tfrac{a}{c}<1.\\
  \end{dcases}
\end{align}
The completed harmonic Maass forms $\mathcal M(a,b,c;z)$ and $\mathcal N(a,b,c;z)$ are defined by
\begin{align}
  \mathcal{M}(a,b,c;z)&=2q^{\frac{3a}{2c}\left(1-\frac{a}{c} \right)-\frac{1}{24}}M(a,b,c;z)+\varepsilon(a,b,c;z)-T_2(a,b,c;z), \\
    \mathcal{N}(a,b,c;z)&=4e\left(-\frac{a}{c}k(b,c) + \frac{3b}{2c}\left(\frac{2a}{c}-1\right)-\frac bc\right) q^{\frac{b}{c}k(b,c)-\frac{3b^2}{2c^2}-\frac{1}{24}}N(a,b,c;z)-T_1(a,b,c;z),
\end{align}
where $T_1(a,b,c;z)$ and $T_2(a,b,c;z)$ are defined in Section~3 of \cite{garvan2019transformation}.
We extend these definitions to the case $a=b=0$ by setting
\begin{equation}
    \mathcal M(0,0,c;z) = \mathcal N(0,0,c;z) = 0.
\end{equation}
By Theorem~3.1 of \cite{garvan2019transformation}, the functions $\mathcal M(a,b,c;z)$ and $\mathcal N(a,b,c;z)$ satisfy the transformation laws \eqref{eq:M-T-transform}, \eqref{eq:N-T-transform}, and \eqref{eq:M-N-S-transform} given in the introduction.

\section{Proof of Theorem~\ref{thm:main}}
\label{sec:proof}

That the components of $H_c(z)$ are annihilated by $\Delta_{1/2}$ was shown already by Garvan in \cite{garvan2019transformation}.
The growth condition at $\infty$ follows from the Fourier expansion of the functions $\mathcal M(a,b,c;z)$ and $\mathcal N(a,b,c;z)$.
So it suffices to show that $H_c(z)$ transforms correctly under the action of $\Mp_2(\Z)$.

Let $L(c)$ be the lattice $\Z$ together with the bilinear form $(x,y) = -12c^2xy$.
Then $L'(c)/L(c)=\{\frac{h}{12c^2} : h\in \Z/12c^2\Z\}$. For convenience we write $\mathfrak{e}_h$ instead of $\mathfrak{e}_{h/12c^2}$. For $L(c)$ equations \eqref{eq:weil-T} and \eqref{eq:weil-S} become
\begin{align}
\rho_{L(c)}(T)\mathfrak{e}_h& = e\left(-\frac{h^2}{24c^2}\right)\mathfrak{e}_h \label{eq:weil-T-c},
\\
    \rho_{L(c)}(S)\mathfrak{e}_h &= \frac{1}{c\sqrt{-12i}}\sum_{h'(12c^2)}e\left(\frac{hh'}{12c^2}\right)\mathfrak{e}_{h'}. \label{eq:weil-S-c}
    \end{align}
By the discussion in Section~\ref{sec:weil-rep} it suffices to verify that
\begin{equation} \label{eq:Hp-T}
    H_c(z+1) = \rho_{L(c)}(T) H_c(z)
\end{equation}
and
\begin{equation} \label{eq:Hp-S}
    H_c(-1/z) = \sqrt{z} \, \rho_{L(c)}(S) H_c(z).
\end{equation}
Because $\rho_{L(c)}(S^2)\mathfrak{e}_h = i\mathfrak{e}_{-h}$, the consistency conditions $\alpha_h(a,b) = -\alpha_{-h}(a,b)$ and $\beta_h(a,b)=-\beta_{-h}(a,b)$ must be satisfied for all $h,a,b.$

\subsection{The transformation (\ref{eq:Hp-T})}
Recall that $[x]_c$ denotes the least nonnegative residue of $x$ modulo $c$.
After applying the equations \eqref{eq:M-T-transform}, \eqref{eq:N-T-transform}, and \eqref{eq:weil-T-c}, we see that \eqref{eq:Hp-T} is implied by the equations
\begin{equation} \label{eq:alpha-T}
    \alpha_h(a,[a+b]_c) = e\left(\frac{h^2}{24c^2} + \frac{5a}{2c} - \frac{3a^2}{2c^2} - \frac{1}{24}\right)\alpha_h(a,b)
\end{equation}
and
\begin{equation} \label{eq:beta-T}
    \beta_h([a-b]_c,b) = e\left(\frac{h^2}{24c^2} + \frac{3b^2}{2c^2} - \frac{1}{24}\right) \beta_h(a,b) \times
    \begin{cases}
    1 & \text{ if $a\geq b$,} \\
    e\left(\frac 12 - \frac{3b}{c}\right) & \text{ if $a<b$.}
    \end{cases}
\end{equation}

For the first equation, we may assume that $h=\pm 6a+ck$ for some $k\in \Z$ with $(k,6)=1$, because $\alpha_h(a,b)=\alpha_h(a,[a+b]_c)=0$ otherwise.
Under that assumption we have
\begin{equation}
    e\left(\frac{h^2}{24c^2} + \frac{5a}{2c} - \frac{3a^2}{2c^2} - \frac{1}{24}\right) = e\left( \frac{a(5\pm k)}{2c}\right). 
\end{equation}
This is enough to prove \eqref{eq:alpha-T} when $a=0$.
Since $5\pm k$ is even, we have
\begin{equation}
    e\left(\frac{[a+b]_c(5\pm k)}{2c}\right) = e\left(\frac{(a+b)(5\pm k)/2}{c}\right) = e\left( \frac{a(5\pm k)}{2c} + \frac{b(5\pm k)}{2c}\right),
\end{equation}
and this is enough to prove \eqref{eq:alpha-T} when $a\neq 0$.

To prove \eqref{eq:beta-T} we may assume that $h=\mp 6b+ck$ for some $k\in \Z$ with $(k,6)=1$.
Then
\begin{equation}
    e\left(\frac{h^2}{24c^2} + \frac{3b^2}{2c^2} - \frac{1}{24}\right) = e\left(\frac{3b^2}{c^2} \mp \frac{bk}{2c}\right).
\end{equation}
This is enough to finish the proof in the case $b=0$ since then $a\geq b$.
When $b\neq 0$ we must show that
\begin{equation}
    e\left(\pm\frac{[a-b]_c k}{2c} - \frac{3[a-b]_cb}{c^2}\right) = e\left(\frac{3b^2}{c^2} \mp \frac{bk}{2c} \pm\frac{ak}{2c} - \frac{3ab}{c^2} \right) \times
    \begin{cases}
    1 & \text{ if $a\geq b$,} \\
    e\left(\frac 12 - \frac{3b}{c}\right) & \text{ if $a<b$.}
    \end{cases}
    \label{eq:beta-T-1}
\end{equation}
If $a\geq b$ then $[a-b]_c=a-b$ and it is straightforward to check that \eqref{eq:beta-T-1} holds.
If $a<b$ then $[a-b]_c = a-b+c$, and thus \eqref{eq:beta-T-1} is true because $k$ is odd.

\subsection{The transformation (\ref{eq:Hp-S})}
As before, we reduce the proof of \eqref{eq:Hp-S} to the verification of two identities by applying \eqref{eq:M-N-S-transform} and \eqref{eq:weil-S-c}.
One identity is
\begin{equation} \label{eq:S-ident}
    \frac{i}{c\sqrt{12}} \sum_{h'(12c^2)} e\pfrac{hh'}{12c^2} \alpha_{h'}(a,b) = \beta_h(a,b),
\end{equation}
and the other is obtained by exchanging the roles of $\alpha_h(a,b)$ and $\beta_h(a,b)$.
If \eqref{eq:S-ident} is true, then
\begin{align}
    \frac{i}{c\sqrt{12}} \sum_{h'(12c^2)} e\pfrac{hh'}{12c^2} \beta_{h'}(a,b) 
    &= \frac{-1}{12c^2} \sum_{h'(12c^2)} e\pfrac{hh'}{12c^2}
    \sum_{h''(12c^2)} e\pfrac{h'h''}{12c^2} \alpha_{h''}(a,b) \\
    &= \frac{-1}{12c^2} \sum_{h''(12c^2)} \alpha_{h''}(a,b) 
    \sum_{h'(12c^2)} e\pfrac{h'(h+h'')}{12c^2}.
    \label{eq:S-two-sums}
\end{align}
The inner sum above equals zero unless $h\equiv -h''\pmod{12c^2}$, in which case it equals $12c^2$.
It follows that \eqref{eq:S-two-sums} evaluates to $-\alpha_{-h}(a,b) = \alpha_h(a,b)$.
Thus it suffices to prove \eqref{eq:S-ident}.
We first prove a lemma.

\begin{lemma} \label{lem:12p-gauss}
If $(c,6)=1$ then
\begin{equation}
    \sum_{k(12c)} \pfrac{12}{k} e\pfrac{k\ell}{12c} = 
    \begin{cases}
    c\sqrt{12} \pfrac{12}{\ell/c} & \text{ if }c\mid \ell, \\
    0 & \text{ otherwise.}
    \end{cases}
\end{equation}
\end{lemma}

\begin{proof}
By the Chinese Remainder Theorem, we can write $k=12x+cy$, where $x$ runs over $\Z/c\Z$ and $y$ runs over $\Z/12\Z$.
Then
\begin{equation}
    \sum_{k(12c)} \pfrac{12}{k} e\pfrac{k\ell}{12c}
    =
    \sum_{x(c)} \sum_{y(12)} \pfrac{12}{cy} e\left(\frac{x\ell}{c}+\frac{y\ell}{12}\right)
    =
    \pfrac{12}{c}\sum_{x(c)} e\pfrac{x\ell}{c} \sum_{y(12)} \pfrac{12}{y} e\pfrac{y\ell}{12}.
\end{equation}
The $x$ sum evaluates to $c$ when $c\mid \ell$, zero otherwise, and the $y$ sum is a Gauss sum that evaluates to $\sqrt{12}\pfrac{12}{\ell}$.
Finally, if $c\mid \ell$ then $\pfrac{12}{c\ell} = \pfrac{12}{\ell/c}$.
\end{proof}

We now prove \eqref{eq:S-ident}. 
If $a\neq 0$ then 
\begin{align}
    \frac{i}{c\sqrt{12}} \sum_{h'(12c^2)} e\pfrac{hh'}{12c^2} \alpha_{h'}(a,b) &=
    \frac{i}{2c\sqrt{12}\sin(\pi/c)} e\pfrac{5b}{2c} \sum_{\ep=\pm 1} \ep \sum_{\substack{h'(12c^2) \\ h'=6\ep a+ck}} \pfrac{12}{k} e\left(\frac{hh'}{12c^2} + \frac{\ep b k}{2c}\right) \\
    &= \frac{i}{2c\sqrt{12}\sin(\pi/c)} e\pfrac{5b}{2c} \sum_{\ep=\pm 1} \ep \, e\pfrac{\ep ha}{2c^2} \sum_{k(12c)} \pfrac{12}{k} e\pfrac{k(h+6\ep b)}{12c}.
\end{align}
By Lemma~\ref{lem:12p-gauss} the latter expression equals
\begin{multline}
    \frac{i}{2\sin(\pi/c)} e\pfrac{5b}{2c} \sum_{\substack{\ep=\pm 1 \\ h= -6\ep b+cr}} \pfrac{12}{r} \ep \, e\pfrac{\ep h a}{2c^2} 
    \\=
    \begin{dcases}
    \pm \frac{i}{2\sin(\pi/c)} \pfrac{12}{r} e\left(\frac{5b\pm ra}{2c} - \frac{3 ab}{c^2} \right) & \text{ if }b\neq 0 \text{ and } h=\mp 6b+cr, \\
    \frac{i}{2\sin(\pi/c)} \pfrac{12}{r} \left(e\pfrac{ra}{2c} - e\left(-\frac{ra}{2c}\right)\right)& \text{ if }b=0 \text{ and }h=cr, \\
    0 & \text{ otherwise.}
    \end{dcases}
\end{multline}
This agrees with the definition of $\beta_h(a,b)$ because $e(\theta) - e(-\theta) = 2i\sin(2\pi\theta)$.

If $a=0$, we may assume that $b\neq 0$. Then
\begin{equation}
    \frac{i}{c\sqrt{12}} \sum_{h'(12c^2)} e\pfrac{hh'}{12c^2} \alpha_{h'}(0,b)
    =
    \frac{-1}{c\sqrt{12}\,\sin(\pi/c)} e\pfrac{5b}{2c} \sum_{k(12c)} \pfrac{12}{k}e\pfrac{hk}{12c} \sin(bk\pi/c).
\end{equation}
Writing $\sin(\pi\theta) =-\frac{i}{2}( e(\theta/2) - e(-\theta/2))$, we find that the latter expression equals
\begin{equation}
    \frac{i}{2c\sqrt{12}\,\sin(\pi/c)}e\pfrac{5b}{2c}\left(
    \sum_{k(12c)}\pfrac{12}{k} e\pfrac{k(h+6b)}{12c}
    -
    \sum_{k(12c)}\pfrac{12}{k} e\pfrac{k(h-6b)}{12c}
    \right).
\end{equation}
Each of the sums above can be evaluated by Lemma~\ref{lem:12p-gauss}.
Since $b\neq 0$, at most one of the sums is nonzero.
If either is nonzero then we have $h=\mp 6b+cr$ for some $r\in \Z$. 
In that case,
\begin{equation}
    \frac{i}{c\sqrt{12}} \sum_{h'(12c^2)} e\pfrac{hh'}{12c^2} \alpha_{h'}(0,b)
    =
    \frac{\pm i}{2\sin(\pi/c)} \pfrac{12}{r}e\pfrac{5b}{2c}.
\end{equation}
This completes the proof of Theorem~\ref{thm:main}.\qed

\section{Motivation}
\label{sec:motivation}

In this section we briefly describe a method for determining the coefficients $\alpha_h(a,b)$ and $\beta_h(a,b)$ in Theorem~\ref{thm:main}.
For simplicity we work in the special case where $c$ is a prime $p\geq 5$.

Assume that equations \eqref{eq:alpha-T}, \eqref{eq:beta-T}, and \eqref{eq:S-ident} hold.
When $a=0$, equation \eqref{eq:alpha-T} shows that $\alpha_h(0,b)$ is nonzero only if $e(\frac{h^2}{24p^2} - \frac 1{24})=1$, that is, if $h=pk$ for some integer $k$ coprime to $6$.
When $a\neq 0$, we apply \eqref{eq:alpha-T} $p$ times to conclude that
\begin{equation}
    \alpha_h(a,[pa+b]_p) = e\left(\frac{h^2}{24p} + \frac{5a}{2} - \frac{3a^2}{2p} - \frac{p}{24}\right)\alpha_h(a,b).
\end{equation}
Since $[pa+b]_p = b$, we see that $\alpha_h(a,b)$ is nonzero only if
\begin{equation}
    h^2 + 12pa-36a^2-p^2\equiv 0\pmod{24p}.
\end{equation}
Note $e(5a/2) = e(a/2)$. By considering this congruence modulo $p$, then modulo $6$, we conclude that $h=\pm 6a+pk$ for some integer $k$ with $(k,6)=1$. A similar statement holds for the $\beta_h$ coefficients. 

With these reductions it is now reasonable to use a computer to numerically solve the (still quite large) system of equations \eqref{eq:alpha-T}, \eqref{eq:beta-T}, and \eqref{eq:S-ident} for a few specific values of $p$, say $p=5$ and $p=7$.
From the numerical solution, one can guess the general form of $\alpha_h(0,b)$ and $\beta_h(a,0)$ since these are very simple.
By using \eqref{eq:S-ident} we can now determine formulas for $\alpha_h(a,0)$ and $\beta_h(0,b)$.
It remains to determine the general form of the other coefficients.

For fixed $a\neq 0$ and $h=\pm 6a+pk$ with $(k,6)=1$, repeated application of \eqref{eq:alpha-T} gives
\begin{equation}
    \alpha_h(a,[ma]_p) = e\left(\frac{ma(5\pm k)}{2p}\right) \alpha_h(a,0).
\end{equation}
If we want $[ma]_p=b$ then we need $m\equiv \bar a b\pmod{p}$, where $\bar a a\equiv 1\pmod{p}$.
Since $5\pm k$ is even, this yields
\begin{align}\label{eq:alpha(a,b)gen}
    \alpha_h(a,b) = e\left(\frac{b(5\pm k)}{2p}\right) \alpha_h(a,0).
\end{align}
Similarly, for fixed $b\neq 0$ and $h=\mp 6b+pk$ with $(k,6)=1$, repeated application of \eqref{eq:beta-T} gives
\begin{equation}
    \beta_h([-mb]_p,b) = e\left(m\left(\frac{3b^2}{p^2}\mp \frac{bk}{2p}\right) +
    \nu(m,b)\left(\frac{1}{2}-\frac{3b}{p}\right)\right)\beta_h(0,b),
\end{equation}
where $\nu(m,b) = \#\{0\leq x<m : [-xb]_p<b\}$.
One can show by induction on $m$ (or by some other means) that for $m\geq 0$ we have
\begin{equation}
    \nu(m,b) = \left\lfloor\frac{mb}{p}\right\rfloor + 
    \begin{cases}
    1 & \text{ if }p\nmid m, \\
    0 & \text{ if }p\mid m.
    \end{cases}
\end{equation}
If $[-mb]_p=a$, then $[mb]_p=p-a$, and it follows that
\begin{equation}
    \nu(m,b) = \frac{mb-[mb]_p}{p} + 1 = \frac{mb+a}{p}. 
\end{equation}
Therefore
\begin{align}\label{eq: beta(a,b) in terms of beta(0,b)}
    \beta_h(a,b) = e\left(\pm \frac{ak}{2p} - \frac{3ab}{p^2}\right) \beta_h(0,b).
\end{align}
In this way, we obtain formulas for all of the coefficients $\alpha_h(a,b)$ and $\beta_h(a,b)$.

\bibliographystyle{alpha}
\bibliography{masterref}

\end{document}